\documentclass[12pt, a4paper]{amsart}
\usepackage{amsfonts,amsmath, amsthm, amssymb}
\usepackage{latexsym}
\usepackage{enumitem}

\usepackage[dvips]{graphicx}
\usepackage{psfrag}
\usepackage{xcolor}
\usepackage{tikz}
\usetikzlibrary{shapes,arrows,positioning,intersections}
\usepackage{hyperref}

\newtheorem{theorem}{Theorem}[section]
\newtheorem{lemma}[theorem]{Lemma}
\newtheorem{corollary}[theorem]{Corollary}
\newtheorem{proposition}[theorem]{Proposition}

\newtheorem{definition}[theorem]{Definition}
\newtheorem{remark}[theorem]{Remark}

\newcommand{\useonX}{} 
\newcommand{\useonM}{\underline}

\newcommand{\tm}{\useonX{m}}
\newcommand{\tB}{\useonX{B}}
\newcommand{\tS}{\useonX{S}}
\newcommand{\tc}{\useonX{c}}
\newcommand{\tv}{\useonX{v}}
\newcommand{\tw}{\useonX{w}}
\newcommand{\hB}{\useonM{B}}
\newcommand{\hS}{\useonM{S}}

\newcommand{\hm}{\useonM{m}}

\newcommand{\Pa}{\mathbf{P}}
\newcommand{\Fu}{\mathbf{F}}

\newcommand{\sqbd}{\partial^2 X}
\newcommand {\R}{\mathbb{R}} 

\DeclareMathOperator{\diam}{diam}

\DeclareMathOperator{\pr}{pr}

\makeatletter
\@namedef{subjclassname@2020}{\textup{2020} Mathematics Subject Classification}
\makeatother

\begin{document}

\title[Counting geodesic loops]{Counting geodesic loops on  surfaces of genus at least $2$ without conjugate points}

\author{Mark Pollicott and Khadim War}

\date{\today}


\maketitle

\begin{abstract}
In this paper we prove asymptotic estimates for closed geodesic loops on 
compact  surfaces with no conjugate points.  These generalize  the classical 
 counting results of Huber and Margulis and sector theorems {\color{black} for surfaces of strictly negative curvature. We will also prove more general sector theorems, generalizing results of Nicholls and Sharp  for special case of surfaces of strictly negative curvature.}
\end{abstract}

\section{Introduction}

For a closed surface $M$ of negative curvature there are classical results which count the number of geodesic arcs starting and ending at a given reference point $p \in M$ and whose length at most $t$, say. 
For constant curvature surfaces these were proved by Huber in 1959, and  for variable curvature surfaces these were proved by Margulis in  1969,   In particular, they give simple asymptotic estimates for this counting function as $t \to +\infty$.  In this brief note we  will extend these results in Corollary \ref{counting} to the more general setting of surfaces without  conjugate points.

 There are refinements of the original counting results of Huber and Margulis whereby the geodesics are restricted to lie in a sector.  These are due to shown for constant curvature surfaces  by Nicholls in  1983, and  for variable curvature surfaces by Sharp in  $2001$. 
  We will describe generalizations  of these results to surfaces without  conjugate points in Corollaries \ref{counting0} and \ref{counting}. These will follow from a more general statement (Theorem \ref{thm:main}) which appears below. 

We begin with some general notation.
Let  $(M,g)$ be a closed Riemannian manifold, $SM$ the unit tangent bundle of $M$ and 
let $\pi: SM\to M$ be the natural projection to the footpoint. 

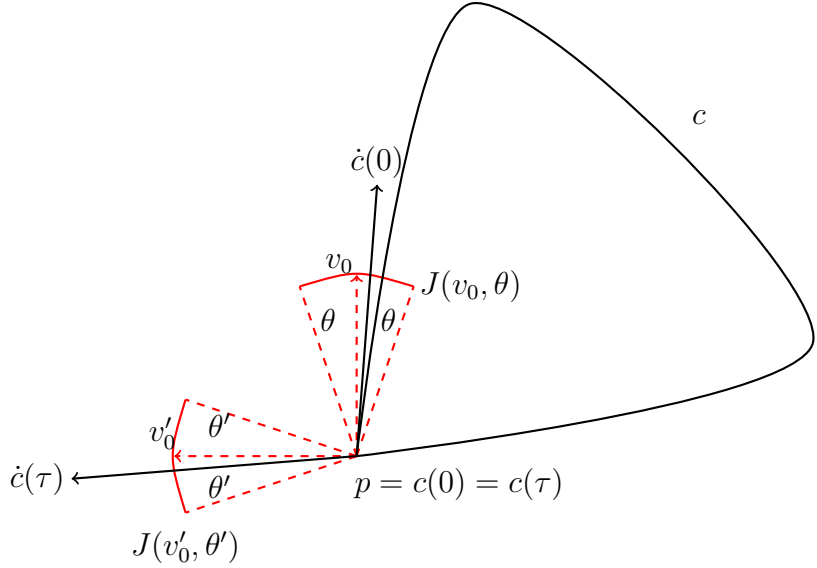
\begin{figure}
\begin{tikzpicture}[thick,scale=1.5, every node/.style={scale=1.0}]
\draw[dashed,red] (0,0) -- (-0.5,1.5) ;
\draw[dashed,red] (0,0) -- (0.5,1.5) ;
\draw[red] (-0.5,1.5) .. controls (0,1.65) .. (0.5,1.5);
\draw[thick, ->,red,dashed] (0,0) -- (0,1.6) ;
\draw[thick, ->] (0,0) -- (0.18,2.4) ;
\node at (-0.25,1.2) {$\theta$};
\node at (0.28,1.2) {$\theta$};
\node at (-0.15, 1.7) {$v_0$};
\node at (1.0, 1.5) {$J(v_0,\theta)$};
\node at (-1.5, -0.8) {$J(v_0',\theta')$};
\draw[dashed,red] (0,0) -- (-1.5,0.5) ;
\draw[dashed,red] (0,0) -- (-1.5,-0.5) ;
\draw[thick, ->,red,dashed] (0,0) -- (-1.6,0) ;
\draw[red]  (-1.5,0.5) .. controls (-1.65,0) ..  (-1.5,-0.5);
\node at (-1.7,0.2) {$v_0'$};
\node at (-1.2, 0.27) {$\theta'$};
\node at (-1.2, -0.27) {$\theta'$};
\node at (3,3) {$c$};
\node at (0.18,2.6) {$\dot c(0)$};
\node at (-2.8,-0.2) {$\dot c(\tau)$};
\draw[thick, ->] (0,0) -- (-2.5,-0.2);
\node at (0.9,-0.25) {$p = c(0) = c(\tau)$};
\draw plot [smooth] coordinates {(0,0)  (1,4)  (4,1)  (0,0)};
\end{tikzpicture}
\caption{A geodesic loop $c$ which starts within an angle $\theta$ of $v_0$ and ends within an angle $\theta'$ of $v_0'$}
\end{figure}

Let  $t,\theta,\theta'>0$ and $v_0,v_0'\in SM$ with $\pi v_0=\pi v_0' = p$, say.
We want to count 
 geodesic loops  $c: [0, \tau] \to M$ which:
\begin{enumerate}
\item
start and finish at $p$  (i.e., $c(0) = c(\tau) = p$);
\item 
have  length  $\tau$ less than  $t$;
\item 
leaves the fibre $S_{p}M$ at an angle at most $\theta$ to $v_0$; and 
\item 
enters the fibre $S_{p}M$ at an angle at most $\theta'$ to $v_0'$.
\end{enumerate}
(see Figure 1).

\begin{definition}
Given an angle $0 < \theta \leq  \pi$  and a unit tangent vector $v_0\in SM$, we define the following arc in the fibre $S_{\pi v_0}M$:     
$$
J(v_0,\theta):=\{w\in S_{p}M: \measuredangle_{p}(v_0,w)\leq\theta\},
$$
i.e., the unit tangent vectors $w$  in the same fibre as $v_0$ at an angle at most  $\theta$.
\end{definition}

This allows us to introduce convenient notation for  the collection of geodesic arcs satisfying properties  (1)--(4).

\begin{definition}\label{def:c}
We let $\mathcal C(t, J(v_0,  \theta), J(v_0',\theta'))$ denote  the set of  geodesic loops $c:[0,\tau]\to M$  based at  
$c(0)=c(\tau)=
p  \in M$ of length  
$\tau \leq t$   and satisfying $c'(0)\in J(v_0,\theta)$ and  $c'(\tau)\in J(v_0',\theta')$.
\end{definition}

We will now  consider 
the problem 
of counting geodesic
the number 
$$\#\mathcal C(t,J(v,\theta), J(v',\theta'))$$
of such geodesic arcs.

We will work in the  general setting  of closed surfaces $M$ 
of genus at least $2$ that 
have  no conjugate points, i.e., for any two points $p,q \in M$ there is no 
geodesic from $p$ to $q$ along which there is a non-trivial Jacobi field  vanishing at $p$ and $q$.
\footnote{Recall that a Jacobi vector field is a  vector field $J(t)$ on the geodesic $c(t)$ satisfying $\frac{D^2}{dt^2} J(t) = R(J(t)), \dot \gamma(t), \dot \gamma(t) = 0$, where $D$ denotes the covariant derivative with respect to the Levi-Civita connection, $R$ the Riemann curvature tensor $\dot \gamma(t) = d\gamma(t)/dt$ the tangent vector field.}
By the Cartan-Hadamard theorem, an  equivalent formulation  is that there is a unique geodesic arc joining distinct points in the universal cover $\widetilde M$.
Examples include 
 the special case that $M$ has non-positive curvature.  We refer to \cite{BBB} for another well known example.

Finally, using  the following notation
$$
S^2M:=\{(v,v')\in SM\times SM: \pi v=\pi v'\}
$$
we can formulate our main result.  

\begin{theorem}\label{thm:main}
Let 
$M$
be a closed connected surface of genus at least $2$ without conjugate points. Then there exists $\theta_0>0$, $h>0$ and a measurable positive function $a:S^2M\times (0,\theta_0)^2\to\mathbb{R}_{>0}$ such that 
\begin{equation}\label{eqn:margulis}
\#\mathcal C(t,J(v,\theta), J(v',\theta')) \sim a(v, v', \theta, \theta')e^{ht}, \hbox{ as } t \to +\infty
\end{equation}
i.e., 
$\lim_{t \to +\infty} \frac{\#\mathcal C(t,J(v,\theta), J(v',\theta'))}{a(v, v', \theta, \theta')e^{ht}}=1$.

\noindent
Moreover if the geodesic flow is expansive
\footnote{A flow $\phi_t:SM \to SM$ is \emph{expansive}
if  for all $\delta >0$ there exists $\epsilon > 0$ 
such that  if  $d(\phi_t(x), \phi_{s(t)}(y)) < \delta$ for all $t \in \mathbb R$ for $x,y \in SM$ and a continuous map $s: \mathbb R \to \mathbb R$ then $y = \phi_t(x)$ where $|t|<\epsilon$.
 }
 then the function $a(\cdot, \cdot,\cdot, \cdot )$ is continuous.
\end{theorem}

In the statement of the theorem the value  $h$ is the topological entropy of the geodesic flow on the unit tangent bundle $SM$.

\begin{remark}
In the special case that $M$ has constant curvature then $a(\cdot)$ is a constant function,  and when $M$ has   variable  negative curvature it is known that $a(\cdot)$ is a continuous function (not least because it is expansive).
\end{remark}


Theorem \ref{thm:main} has corollaries which extend several classical results from  the context of negative curvature.
In particular, this leads to generalizations of classical counting and sector theorems.
For example, when  we set $\theta'=\pi$ then this gives the following.

\begin{corollary}[Sector Theorem]\label{counting0}
Given $0 < \theta \leq  \pi$ there exists $a=a(p, \theta)>0$ such that  the number of 
geodesic arcs 
  which:
start at $p \in M$  and finish at $q \in M$;
leave  $S_{p}M$ at an angle at most $\theta$ to $v_0$;
and have  length at most $t$,   is asymptotic to 
$a e^{h t}$ as $t \to +\infty$.
\end{corollary}

This generalizes results  from \cite{Ni1}, \cite{Ni2}, \cite{Sh}.

Furthermore, when   $\theta = \theta' =\pi$ then this further reduces to the original   counting result:

\begin{corollary}[Arc counting]\label{counting}
There exists $a=a(p)>0$ such that the  number of 
geodesic arcs 
  which
start  at  $p \in M$, finish at $q \in M$ and have   length at most $t$  is asymptotic to 
$a e^{h t}$ as $t \to +\infty$.
\end{corollary}

This generalizes results from
\cite{Hu},
 \cite{Ma1}, \cite{Ma2}.

{\color{black}

Finally, we can describe  equidistribution result  of a slightly different flavour.
Let $\widehat M$ be a finite cover for $M$.   We can associate to any geodesic arc $c$ on $M$ which starts and ends at $p \in M$ (and has length $L_c$) a lift $\widehat c$ to $\widehat M$.  The following corollary estimates the proportion of  geodesic arcs such that  $\widehat c$ on 
$\widehat M$  with  $\widehat c(0) = \widehat c(L_c)$

\begin{corollary}[Equidistribution in finite covers]
The  proportion  of 
geodesic arcs $c$  which
start and end  at $p \in M$, have lifts $\widehat c$ which start and end at the same point in $\widehat M$,   and  have   length at most $t$  is asymptotic to 
$$\frac{\hbox{\rm Area}(M)}{ \hbox{\rm Area} (\widehat M)} a e^{h t} \hbox{ as $t \to +\infty$ }
$$

\end{corollary}

This corollary can be used to prove a corollary related  to the first Homology group $H_1(M, \mathbb Z)$. 
Each closed loop $c$ based at $p$ gives rise naturally  to an element $\langle c\rangle \in H_1(M, \mathbb Z)$.  Let us consider a 
finite index  subgroup $G <  H_1(M, \mathbb Z)$ then for a geodesic arc $c$ we can associate the  coset  $\langle c\rangle G \in   H_1(M, \mathbb Z)G$.

\begin{corollary}[Homological Equidistribution]
For a fixed  coset $\alpha \in H_1(M, \mathbb Z)/G$.
The  proportion  of 
geodesic arcs $c$  which
start and finish  at $p \in M$, satisfy $\langle c \rangle \Gamma = \alpha$ and  have   length at most $t$  is asymptotic to 
$$\#(H_1(M, \mathbb Z)/\Gamma) a e^{h t} \hbox{ as } t \to +\infty 
$$
 \end{corollary}
}

\smallskip
\begin{remark}
The theorem and each of the corollaries has a natural  equivalent formulation in terms of the action 
$\Gamma \times X \to X$
of the covering group $\Gamma = \pi_1(M)$ on the universal cover $X$.  For example, Corollary \ref{counting} gives an asymptotic estimate for $\#\{g \in \Gamma \hbox{ : } d_X(\overline p, g \overline p) \leq t\}$ where $\overline p \in X$ and $d_X$ is the lifted Riemannian metric to $X$.
\end{remark}

\section{Closed arcs and isometries}
The structure of the proof of Theorem \ref{thm:main} follows the lines of Margulis' original proof.  However, it requires modifications
using a number of  recent  techniques from \cite{CKW1}, \cite{CKW2}.
A  key ingredient is the  construction of the measure of maximal entropy for the geodesic flow $\phi_t: SM \to SM$.

\subsection{Some Notation}
Let $X$ be the universal cover of $M$ with the lifted metric.  The covering group $\Gamma \cong \pi_1(M)$ satisfies that $M = X/\Gamma$.

Let $SX$ denote  the unit tangent bundle for $X$   and let $\overline \pi: SX \to X$ denote  the canonical projection of a unit tangent vector in $SX$ to its footpoint in $X$.
Let $\overline p \in X$ be a lift of $p \in M$ and let   $\underline B(\overline p, R) \subset X$ denote  a ball of radius $R>0$ about  $\overline p$. 
We can use this to give a  convenient  definition of topological entropy \cite{FM}.

\begin{definition}
The \emph{topological entropy} $h = h(\phi)$ is given by
$$h = \lim_{R\to +\infty} \frac{\log \hbox{\rm Vol} (\underline B(\overline p,R))}{R}.$$
\end{definition}

Given  $\overline v\in SX$, let $c= c_{\overline  v}:\mathbb{R}\to X$ denote the unique geodesic such that $c_{\overline  v}(0)=\overline \pi(\overline  v)$ and $c'_{\overline  v}(0)= \overline  v$. 

\medskip
\begin{definition}
Let $\partial X$ denote  the ideal boundary of $X$ consisting of equivalence classes $[c]$ of geodesics $c : \mathbb R \to X$ which stay a bounded distance apart.
\end{definition} 
\noindent
(See \cite[Section 2]{CKW1} for a detailed description of  the  construction and properties of $\partial X$).
In particular, every geodesic $c_{\overline v}:\mathbb{R}\to X$ defines two points $c(\pm\infty)\in\partial X$, which it is  convenient to denote $\overline v^-:=c(\-\infty)$
and $\overline v^+ :=c(+\infty)$.
The natural acton $\Gamma \times X \to X$ extends to an action 
 of $\Gamma$ on $\partial X$ given by $g[c] = [g c]$, where $g \in \Gamma$.
\medskip
\noindent


\begin{definition}
Given $\overline p \in X$, 
the \emph{Busemann function}
 $b_{\overline p}(\cdot, \cdot): X \times \partial X \to \mathbb R$   is defined by 
 $$b_{\overline p}(\overline q, \xi) = \lim_{t \to +\infty} d(\overline q, c_v(t))-t$$ for  $\overline v\in S_{\overline p}X$ satisfying 
  $\xi = c_{\overline v}(+\infty)$ \cite[Definition 2.16]{CKW1}. 
\end{definition}

\medskip
We next recall the characterization  of Patterson-Sullivan measures on the boundary $\partial X$ 
 constructed in \cite[Proposition 5.1]{CKW1}.

\begin{definition}
The \emph{Patterson-Sullivan measures} on $\partial X$ 
are a family of measures $\{\mu_{\overline  p} \hbox{ : }  \overline p\in X\}$ 
which  transform under the action of $\Gamma$ on $\partial X$ by
$$\frac{d\mu_{{\overline p}}\gamma}{d\mu_{\overline p}}(\xi) = e^{-h b_{\overline p}(\gamma {\overline p}, \xi)}$$
 for $\gamma \in \Gamma$ and $\xi \in \partial X$
  \end{definition}
  
  The Busemann function is also used in defining horocycles.
  
  \begin{definition}
   The
  \emph{stable horocycle} is defined by
  $$H_\xi({\overline p}) = \{\overline q \in X \hbox{  : } b_{\overline p}(\overline q, \xi) = 0\}$$  
  and the  \emph{unstable horocycle}  is defined by
  $$H^-_\xi({\overline p}) = \{q \in X \hbox{  : } b_{\overline p}(\overline q, -\xi) = 0\}$$
   where $-\xi$ is the antipodal vector to $\xi$. 
  \end{definition}

Finally, we define a class of tangent vectors which will serve us well in the proof.

\begin{definition}
We denote by  $\mathcal E \subset SX$  the set of \emph{expansive vectors} consisting of those unit tangent vectors whose stable and unstable 
horocycles
intersect at exactly one point.
\end{definition}

\begin{figure}[h!]
\centerline{
\begin{tikzpicture}[thick,scale=0.8, every node/.style={scale=0.8}]
\draw (0,0) circle (80pt);
\draw (1.5,0) circle (37pt);
\node at (3.1,0) {$\xi$};
\filldraw (-1,0) circle (1.5pt);
\node at (-1.5,0) {$q$};
\filldraw (0.8,1.1) circle (1.5pt);
\node at (0.3,1.1) {$\overline p$};
\draw[<->] (-0.95,0)--(0.2,0);
\node at (-0.35,-0.5) {$b_{\overline p}(\overline q, \xi)$};
\node at (1.4,-0.9) {$H_\xi(\overline p)$};
\end{tikzpicture}
\hskip 1cm
\begin{tikzpicture}[thick,scale=0.8, every node/.style={scale=0.8}]
\draw (0,0) circle (80pt);
\draw (1.05,0) circle (50pt);
\draw (-1.05,0) circle (50pt);
\node at (3.1,0) {$\xi$};
\node at (-3.1,0) {$\eta$};
\%filldraw (-1,0) circle (1.5pt);
\filldraw (0,1.4) circle (1.5pt);
\node at (0,1.9) {$\overline p$};
\draw[<->] (-0.70,0)--(0.7,0);
\node at (0.0,0.4) {$b_{\overline p}(\xi, \eta)$};
\node at (1.4,-1.2) {$H_\xi(\overline p)$};
\node at (-1.4,-1.2) {$H_\eta(\overline p)$};
\end{tikzpicture}
}
\caption{ 
(i) The geometric interpretation of the Busemann function as the  signed distance  of $\overline q$ from $H_\xi(\overline p)$ corresponds to $b_p(q,\xi)$; 
(ii) The distance  between the horocycles $H_\xi(\overline p)$ and 
$H_\eta(\overline p)$ corresponds to $b_{\overline p}(\xi, \eta)$}
\end{figure}
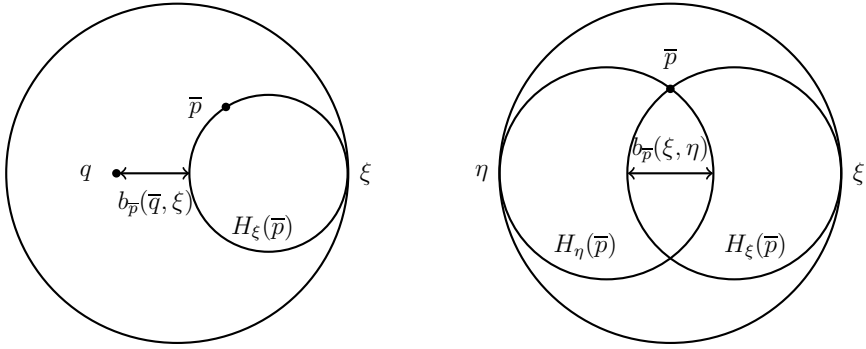

\subsection{The measure of maximal entropy}
We  begin with a correspondence  which is useful in the construction of measures of maximal entropy.

\begin{definition}\label{def:hopf}
The \emph{Hopf map} $H \colon SX \to \sqbd \times \R$  is defined by 
\begin{equation}\label{eqn:hopf}
H(\tv) := ({\overline v}^-, {\overline v}^+, s(\tv))\quad\text{ where }\quad s({\overline v}) :=b_{p}(\pi ({\overline v}, ({\overline v}^-) .
\end{equation}
\end{definition}

In particular, 
following \cite[Lemma 5.5]{CKW1} 
this family of measures defines a $\Gamma$-invariant measure $\overline{\mu}$  on $\partial X\times\partial X\setminus \text{diag}$ 
(where $\text{diag} \subset \partial X\times\partial X$ are the diagonal elements) characterized by  
\begin{equation}\label{eq:mub}
d\overline{\mu}(\xi,\eta)=e^{h\beta_{\overline p}(\xi,\eta)}d\mu_{\overline p}(\xi)d\mu_{\overline p}(\eta), \hbox{ for } \xi, \eta \in \partial X,
\end{equation}
where $\beta_{\overline p}(\xi,\eta)$ is the distance in $X$  between the horospheres
$H_{\overline p}(\xi)$ and $H_{\overline p}(\eta)$, see  Figure 2 (ii) (or \cite[Figure 1]{CKW2}).

\begin{definition}
The Hopf transform carries $d\overline{\mu} \times dt$ to a measure $d\overline m := H_*(d\overline{\mu} \times dt)$
on $SX$. 
\end{definition}

There is a natural projection from $SX$ to $SM$ (taking $v$ to $v \Gamma$).
The following result was proved in \cite[Theorem 1.1]{CKW1}.

\begin{lemma}
The measure $\overline m$ on $SX$
projects 
(after normalization)
to the  measure 
$\underline m$ 
maximal entropy for the geodesic flow on $SM$ (i.e., $\pi_*\overline m = \underline m$ and  $h(\underline m) = h$).   Moreover,
\begin{enumerate}
\item 
$\underline m$ is unique, strongly mixing
\footnote{$\underline m$ is even shown to be Bernoulli}
  and fully supported; and 
\item
$
\underline 
m(\mathcal E) =1$
 (cf. \cite[Equation (2.10)]{CKW1}). 
\end{enumerate}
\end{lemma}

We now turn to the final ingredients in the proof.
\subsection{Flow boxes}
For the remainder of this section, we fix a choice of $(v_0,v_0')\in S^2M\cap \mathcal E^2$.
We can then  associate to the sets $J(v_0,\theta), J(v_0',\theta') \subset SM$ in 
Definition \ref{def:c} a choice of 
lifts 
 $\overline J(v_0,\theta), \overline J(v_0',\theta') \subset SX$.

 To proceed we want to consider the natural  images of these sets in $\partial X$:

 \begin{definition}
We can  associate to $J(v_0,\theta)$ and $J(v_0',\theta')$ their 
 ``future'' and ``past'' subsets of $\partial X$ defined, respectively,  by 
$$
\Fu=\Fu_\theta := \{  \bar w^+ : \bar w\in \overline  J(v_0,\theta) \} \text{ and }
\Pa=\Pa_\theta := \{ \bar w^- : \bar w\in \overline J(v_0,\theta) \}
$$
$$
\Fu'=\Fu_{\theta'} := \{ \bar w^+ : \bar w\in \overline J(v_0',\theta') \}\text{ and }
\Pa'=\Pa_{\theta'} := \{ \bar w^- : \bar w\in \overline J(v_0',\theta') \}.
$$
\end{definition}

\begin{figure}[h!]
\centerline{
\begin{tikzpicture}[thick,scale=1.0, every node/.style={scale=0.8}]
\draw (0,0) circle (80pt);
\draw[dashed] (0,0) -- (2.3, 1.7) ;
\draw[dashed] (0,0) -- (2.3, -1.7) ;
\draw[dashed] (0,0) -- (-2.3, 1.7) ;
\draw[dashed] (0,0) -- (-2.3, -1.7) ;
\draw[->, red, dashed] (0,0)--(1.25,0);
\node at (1.5,0.2) {$v$};
\node at (3,1) {$\Fu$};
\node at (-3,1) {$\Pa$};
\node at (0.8,0.3) {$\theta$};
\node at (0.8,-0.3) {$\theta$};
\draw[ultra thick,red]  (2.2, 1.7) .. controls (3.0, 1)  and  (3.0, -1)  .. (2.2, -1.7) ;
\draw[ultra thick,blue]  (-2.2, 1.7) .. controls (-3.0, 1)  and  (-3.0, -1)  .. (-2.2, -1.7) ;
\draw[ultra thick,red]  (1.0, 0.8) .. controls (1.4, 0.5)  and  (1.4, -0.5)  .. (1.0, -0.8) ;
\node at (1.7,-0.7) {$\widetilde J(v, \theta)$};
\end{tikzpicture}
}
\caption{
The sets $\Pa$ and $\Fu$ associated to $\widetilde J(v,\theta)$}
\end{figure}
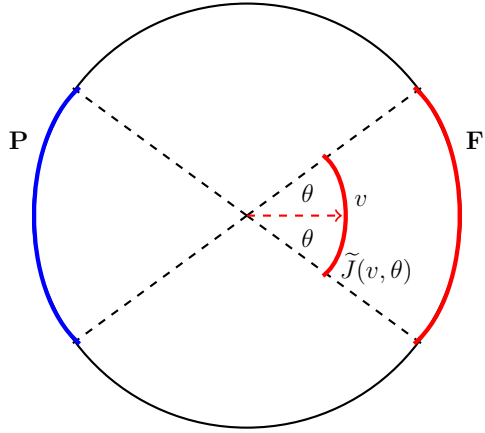

The sets $\Fu, \Pa, \Fu', \Pa' \subset \partial X$ will be used to construct flow  boxes for the geodesic flow.
Assume first that $\epsilon>0$ is small (with respect to the injectivity radius of 
$M$)
and then choose  $\theta_1>0$ such that 
for all $\theta < \theta_1$ we have
$$
\hbox{\rm diam} \left(\pi H^{-1}
(\Pa \times \Fu \times\{0\}) \right) < \frac{\epsilon}{2}
$$
(see \cite[Lemma 3.9]{CKW2}).   For $\alpha\leq\frac{3}{2}\epsilon$ and $\theta\in(0,\theta_1)$ we define two different flow boxes
\footnote{for the geodesic flow  $\phi_t: SX \to SX$ on $SX$}
$\tB_\theta^\alpha$ and $\tB_{\theta'}^{\alpha'}$
 (of different ``lengths'' $\alpha$ and $\epsilon^2$, respectively) in $SX$ by:

\begin{equation}\label{eqn:B}
\begin{aligned}
&\overline \tB_\theta^\alpha := H^{-1}(\Pa\times\Fu \times [0, \alpha])
\hbox{ and } \cr
&\overline \tB_{\theta'}^{\epsilon^2} := H^{-1}(\Pa'\times\Fu' \times [0, \epsilon^2]). 
\end{aligned}
\end{equation}
(cf.   \cite[(3.11) and (3.12)]{CKW2}).

Let $\underline  \tB_\theta^\alpha = \pi(\tB_\theta^\alpha)$
and $\underline \tB_{\theta'}^{\epsilon^2}  = \pi(\tB_{\theta'}^{\epsilon^2} )$ be their projections onto $SM$.

\begin{remark}
Since  the function $\rho'\to \tm(\underline \tB^{\epsilon^2}_{\rho'})$ is nondecreasing, and thus 
 has countably many discontinuties (by Lebesgue's Theorem), we can  suppose without loss of generality  that $\theta' \in (0, \theta_1)$ is a continuity point,  and so, in particular, 
\begin{equation}\label{eqn:SB-cty}
\lim_{\rho'\to\theta'}\tm(\underline \tB^{\epsilon^2}_{\rho'})=\tm(\underline \tB^{\epsilon^2}_{\theta'}).
\end{equation}
\end{remark}

In order to give  a dynamical approach  to the counting problem   the following two definitions will prove useful.
Let $\phi^t: SX \to SX$ denote the geodesic flow on $SX$.

\begin{definition} For $t>0$ we can define two subsets of $\Gamma$ by:
\begin{equation}\label{eqn:Gt}
\begin{aligned}
\Gamma_{\theta,\theta'}(t)
& :=
\{\gamma\in \Gamma : \overline \tB_{\theta'}^{\epsilon^2} \cap \phi^{-t} \gamma_* \overline \tB_\theta^\alpha \neq \emptyset \}
\end{aligned}
\end{equation}
\begin{equation}\label{eqn:Gt*}
\Gamma^*_{\theta,\theta'}(t)  
:=
\{\gamma\in \Gamma_{\theta,\theta'}(t) :  \gamma\Fu\subset\Fu'\text{ and } \gamma^{-1}\Pa\subset\Pa' \}.
\end{equation}
where the sets have an implicit dependence on $\epsilon, \alpha,v_0, v_0'$.
(cf. \cite[(4.4) and (4.14)]{CKW2}.)
\end{definition}

By definition we have $\Gamma_{\theta, \theta'}^*(t) \subset \Gamma_{\theta, \theta'}(t)$ and 
although we may not  expect the reverse inclusion to be true, we have the following slightly more modest result.

\begin{lemma}\label{lem:vis1}
For every $\rho'\in(0,\theta')$ and $\rho\in (0,\theta)$, there exists $t_0>0$ such that 
$$
 \Gamma_{\rho,\rho'}(t)\subset  \Gamma_{\theta,\theta'}^*(t)\quad\text{ for all}\quad t\geq t_0.
$$
\end{lemma}

We postpone the proof of Lemma \ref{lem:vis1} until Appendix A.

The next lemma shows there is an inclusion of the set defined in 
Definition \ref{def:c} into 
$\Gamma(t)$.

\begin{lemma}\label{lem:vis2}
We have an injection
$$
\mathcal C(t,J(v_0,\theta), J(v'_0,\theta'))\hookrightarrow \Gamma(t)
$$
which associates to a geodesic $c$ the associated homotopy class $[c] \in  \pi_1(M)
\cong \Gamma$. 
\end{lemma}

We postpone the proof of Lemma \ref{lem:vis2} until Appendix A.

Although we may not expect the reverse inclusion in Lemma \ref{lem:vis2} to be true, we at least have the following partial   result.

\begin{lemma}\label{lem:vis3}
For every $\rho'\in(0,\theta')$, there exists $t_0>0$ such that there is an inclusion 
$$
\Gamma_{\theta,\rho'}(t)\hookrightarrow \mathcal C(t\pm2\epsilon,J(v_0,\theta), J(v'_0,\theta'))\quad\forall t>t_0.
$$
\end{lemma}

Again we  postpone the proof of Lemma \ref{lem:vis3} until Appendix A.

\section{Proof of the counting results}
In this section we will 
use results from the previous section to 
prove  the following  proposition,  which easily implies Theorem \ref{thm:main}.

\begin{proposition}\label{prop:coun}
We have an asymptotic expression for the cardinality of $\Gamma(t)$ of the form:
\begin{equation}\label{asymp}
\#\Gamma(t)\sim e^{ht}
\overline m(\tB)\frac{\mu_{\overline p}(\Fu')}{\mu_{\overline p}(\Fu)}
\hbox{ as } t \to +\infty.
\end{equation}
Moreover, if the geodesic flow is expansive then the quantity $m(\tB)\frac{\mu_{\overline p}(\Fu')}{\mu_{\overline p}(\Fu)}$ depends continuously on $v, v', \theta, \theta'$.
\end{proposition}

\begin{remark}
The constant on the righthand side of (\ref{asymp}) depends on $p$,
 but not then on the choice of $\bar p \in \pi^{-1}(p)$.
\end{remark}
We begin with a little more notation.  
Let 
\begin{equation}\label{eqn:B'}
S_\theta = H^{-1} \left( \Pa \times \Fu \times [0, \epsilon^2] \right) \subset SX
\end{equation}
be another flow box 
and  let 
$$\Gamma^*(t,\alpha):= 
\{ \gamma \in \Gamma^* \hbox{ : } S_\theta \cap \gamma_*\phi^{-t}B_\theta^\alpha \neq \emptyset\}.$$
%
The proof of Proposition \ref{prop:coun} now depends on  the following two  technical lemmas.

\begin{lemma}\label{lem:full-branch}
For $\gamma\in \Gamma^*(t,\alpha)$, we have
\[
\tB_{\theta'}^{\epsilon^2}\cap \phi^{-(t+2\epsilon^{\frac{3}{2}})}\gamma_* \tB^{\alpha+4\epsilon^{\frac{3}{2}}}_\theta = H^{-1}(\Pa'\times\gamma\Fu\times[0,\epsilon^2])
=: \tS^\gamma.
\]
\end{lemma}


The next lemma describes the $\overline m$-measure of the set $S^\gamma$.

\begin{lemma}\label{lem:scaling}
For each $\gamma \in \Gamma^*$, we have
\[
\overline \tm(\overline \tS^\gamma) =\epsilon^2 e^{\pm 4 h \epsilon} e^{-ht} \mu_p(\Pa')\mu_p(\Fu),
\]
and similarly with $\overline m$ and 
$\overline \tS^\gamma$ on $SX$ replaced by 
the projections 
$m$ and 
$S^\gamma = \pi(\overline S^\gamma)$ onto $SM$.
\end{lemma}

We postpone the proofs of both of these lemmas until Appendix B.

\begin{proof}[Proof of Proposition \ref{prop:coun}]
This follows the general lines of \S 5.2 in \cite{CKW2}.
It follows from Lemmas \ref{lem:vis1} and \ref{lem:full-branch} that given any ${\alpha}\in (0,\frac 32\epsilon]$ and $\rho'\in (0,\theta'), \rho\in(0,\theta)$, for all sufficiently large $t$ we have
\[
\hB^{\epsilon^2}_{\rho'} \cap \phi^{-t} \hB_\theta^\alpha
\subset \bigcup_{\gamma\in \Gamma^*(t,{\alpha})} \hS^\gamma
\subset \hB^{\epsilon^2}_{\theta'} \cap \phi^{-(t+2\epsilon^2)} \hB_\theta^{{\alpha}+4\epsilon^2}
\]
by proving the corresponding result on $SX$ and projecting to $SM$.

Using Lemma \ref{lem:scaling}, 
for all $\gamma\in \Gamma^*(t)$, we have
$$
\begin{aligned}
e^{-4h\epsilon} \hm(\hB^{\epsilon^2}_{\rho'} \cap \phi^{-t} \hB)
&\leq \epsilon^2\#\Gamma^*(t,{\alpha}) e^{-ht}\mu_p(\Pa')\mu_p(\Fu) \cr
&\leq e^{4h\epsilon} \hm(\hB^{\epsilon^2}_{\theta'} \cap \phi^{-(t+2\epsilon^2)} \hB_\theta^{{\alpha}+4\epsilon^2}).
\end{aligned}
$$
Sending $t\to\infty$, using mixing, and dividing through by 
$\hm(\hB^{\epsilon^2}_{\theta'})\hm(\hB_\theta^\alpha) = 
\overline \tm(\tB_{\theta'}^{\epsilon^2}) \overline \tm(\tB_{\theta}^\alpha)$, we get
\[
e^{-4h\epsilon}\frac{\overline \tm(B^{\epsilon^2}_{\rho'})}{\overline \tm(B^{\epsilon^2}_{\theta'})}
\lesssim \frac{\epsilon^2\#\Gamma^*(t,{\alpha})\mu_p(\Pa')\mu_p(\Fu)}{e^{ht} \overline \tm(\tB_{\theta"}^{\epsilon^2}) \overline \tm(\tB_\theta^\alpha)}
\lesssim e^{4h\epsilon} \frac{\overline \tm(\tB_\theta^{{\alpha}+4\epsilon^2})}{\overline \tm(\tB_\theta^\alpha)}.
\]
By \eqref{eqn:SB-cty}, assuming that $\theta'$ is a point of continuity for $\rho'\mapsto \tm(\tB'_{\rho'})$, so we can send $\rho'\nearrow \theta'$ and obtain
\begin{equation}\label{eqn:counting-bounds}
e^{-5h\epsilon}
\lesssim \frac{\#\Gamma^*(t,{\alpha})}{e^{ht}\overline \tm(\tB)}\frac{\mu_p(\Fu)}{\mu_p(\Fu')}
\lesssim e^{5h\epsilon} (1+4\epsilon^2/{\alpha}).
\end{equation}
Finally we need to replace $\# \Gamma^*(t, \alpha)$ by
$\#\Gamma(t)$.
(cf. Compare with \cite[(5.4)]{CKW2}.)

This ends the proof of \ref{asymp}. Finally, if the geodesic flow is expansive  then the space of geodesics is in bijection with  $\sqbd$ then using that the Busemann function $b_p(q,\xi)$ depends continuous on $(p,q,\xi)$ we have $m(\tB)\frac{\mu_{\overline p}(\Fu')}{\mu_{\overline p}(\Fu)}$ depends continuously on $v, v', \theta, \theta'$.
\end{proof}

In order to allow for arbitrary $\theta$ and $\theta'$ in the main theorem
we can break the arcs 
$J(\cdot, \cdot)$ into smaller pieces and apply the proposition.

\appendix

\section{Proofs of lemmas on isometries \\and closed arcs}

This section is devoted to the proof of Lemmas \ref{lem:vis1}, \ref{lem:vis2} and \ref{lem:vis3}. 
The proof of Lemma \ref{lem:vis2} is relatively easy while Lemma \ref{lem:vis1} and \ref{lem:vis3} both uses a geometric featture of surfaces  without conjugate point that we first recall here.

\begin{definition}
A simply connected Riemannian manifold $X$ without conjugate points is a (uniform) visibility manifold if for every $\epsilon>0$ there exists $L>0$ such that whenever a geodesic $c:[a,b]\to X$ stays at a distance at least
$L$ from some point $p \in X$, then the angle sustained by $c$ at $p$ is less than $\epsilon$, that is
$$\measuredangle_p(c) = \sup_{a\leq s,t\leq b} \measuredangle_p((c(s), c(t)) < \epsilon.$$
\end{definition}

\begin{proof}[Proof of Lemma \ref{lem:vis1}]
The proof uses \cite[Lemma 4.9]{CKW2} with the choices  $R=\Fu'_{\rho'}$, $Q=\Pa'_{\rho'}$, $V=int(\Fu'_{\theta'})$ and $U=int(\Pa'_{\theta'})$. 
\end{proof}

\begin{proof}[Proof of Lemma \ref{lem:vis2}]
Let $\underline{c}\in \mathcal C(t,J(v_0,\theta), J(v'_0,\theta'))$ and $c$ be the lift of $\underline{c}$ on $X$ with $\underline c(0)=p$.  There exists $\gamma\in\Gamma$ such that $c(t)=\gamma p=\gamma c(0)$. 
Let $\pr_*: SX \to SM$ be the map associated to $\pi: X \to M$ then 
by definition of $ \mathcal C(t,J(v_0,\theta), J(v'_0,\theta'))$,  for $w=c'(t)$, $\pr_*w\in \tB_{\theta'}^{\epsilon^2}$ and $\phi^{-t}w=c'(0)\in\tB_\theta^\alpha$
implies that $\bar\tw :=\gamma_*\tw\in B_{\theta'}^{\epsilon^2}$ for some $\gamma\in\Gamma$. Therefore  $\bar\tw\in B_{\theta'}^{\epsilon^2}\cap\phi^{-t}\gamma_*B_\theta^\alpha$.
\end{proof}

\footnote{Compare with \cite[(4.8)]{CKW1}}
\begin{proof}[Proof of Lemma \ref{lem:vis3}]
Let $\gamma\in\Gamma_{\theta,\rho'}(t)$ and 
$w\in B^{\epsilon^2}_{\rho'}\cap\phi^{-t}\gamma_*B_\theta^\alpha$. 
By the triangle inequaity
$$
\begin{aligned}
d(p,\gamma p)&\leq d(p,\pi w)+d(\pi w, \pi \phi^tw)+ d(\pi \phi^t w, \gamma p).
\end{aligned}
$$
By \cite[Lemma 3.10 ]{CKW2}, we have $d(p,\pi\tw)\leq \diam(B')\leq2\epsilon$ and $d(\pi\phi^t\tw,\gamma p)\leq \diam(\tB)\leq 2\epsilon$. Substituting these into the above display inequality gives
$$
d(p,\gamma p)\leq t+4\epsilon.
$$

We are left to prove that the geodesic $c:=c_{p,\gamma_p}$ connecting $p$ to $\gamma p$ satisfies $c'(0)\in J(v_0,\theta)$ and $c'(d(p, \gamma p))\in J(v_0',\theta')$.

Let $v\in S_pX$ such that $v^+=w^+\in\Fu$, in particular, there exists $R>0$ such that  $d(c_{v}(t), c_w(t))\leq R$ and therefore the geodesic connecting $\gamma p$ to $c_v(t)$ stays at distance at least $ t-2R$. Then using the uniform visabilty, there 
exists $t_0$ such that for all $t>t_0$  we have $\measuredangle_p(v, c'(0))\leq \theta-\rho$ which implies that $c'(0)\in \Fu$.  Therefore by the uniform visibility, we have $\measuredangle_p(c_{p,\gamma p}'(0), c_{v}'(0))\leq \theta-\rho$, in particular $c'_{p,\gamma p}(0)\in J(v_0,\theta)$. Similarly we use the same visibility condition for the point $\gamma p$ and the geodesic joining $p$ and $c_v(-t)$ where $v\in S_{\gamma p}X$ with $v^{-}=w^{-}$. Thus the geodesic $c_{p,\gamma p}$ belongs to  $\mathcal C(t\pm2\epsilon,J(v_0,\theta), J(v'_0,\theta'))$.
\end{proof}

\section{Counting}
This section is devoted to the proof of Lemmas \ref{lem:full-branch} and \ref{lem:scaling}. The proof uses some geometric quantities that we will define first.

\begin{definition}
For $\xi\in\partial X$ and $\gamma\in\Gamma$, we let $b^{\gamma}_\xi:=b_{\xi}(\gamma p,p)$
\end{definition}

\begin{lemma}\label{lem:s-interval}
Given any $\gamma\in\Gamma^* = \{\gamma \in \Gamma \hbox{ : } \gamma \Fu \subset \Fu \hbox{ and } \gamma^{-1}\Pa \subset \Pa \}$ and any $t\in\R$, we have
\[
\tB_{\theta'}^{\epsilon^2} \cap \phi^{-t} \gamma_* \tB_{\theta}^{\alpha}
= \{ \tw\in E^{-1}(\Pa' \times \gamma \Fu) :
s(\tw) \in [0,\epsilon^2] \cap (b_{\tw^{-}}^\gamma - t + [0,\alpha])\}.
\]
\end{lemma}

\footnote{Compare this to the proof of  \cite[Lemma 4.13]{CKW2}}

\begin{proof}[Proof of Lemma \ref{lem:s-interval}]
To prove that $\tB_{\theta'}^{\epsilon^2}\cap \phi^{-1}\gamma_*\tB_\theta^\alpha\subset E^{-1}(\Pa'\times\gamma\Fu)$, we observe that if $E(\tw) \notin \Pa'\times \gamma\Fu$, then either $\tw^-\notin \Pa'$, so $\tw\notin \tB_{\theta'}^{\epsilon^2}$, or $\tw^+\notin \gamma\Fu$, so 
$\tw\notin \phi^{-t}\gamma_*\tB_\theta^{\alpha}$.

It remains to show that given $\tw\in E^{-1}(\Pa'\times\gamma\Fu)$, we have
\begin{align}
\label{eqn:S-int}
\tw\in \tB_{\theta'}^{\epsilon^2} \ &\Leftrightarrow\ s(\tw) \in [0,\alpha], \text{ and}\\
\label{eqn:gB-int}
\tw\in \phi^{-t} \gamma_* \tB_\theta^\alpha
\ &\Leftrightarrow\ s(\tw) \in b_{\tw^{-}}^\gamma - t + [0,\alpha].
\end{align}
The first of these is immediate from the definition of $\tB'$. For the second, we observe that $s(\tv) = b_{\tv^-}(\pi \tv,p) = b_{\gamma \tv^-}(\gamma\pi \tv,\gamma p)$, and thus
\begin{align*}
\gamma_* \tB
&= \{\gamma_* \tv : \tv\in E^{-1}(\Pa\times\Fu)
\text{ and } b_{\tv^-}(\pi \tv,p) \in [0,\alpha] \} \\
&= \{ \tw \in E^{-1}(\gamma\Pa\times\gamma\Fu) : 
b_{\tw^-}(\pi \tw,\gamma p) \in [0,\alpha]\}
\end{align*}
By \cite[Equation (3.1)]{CKW2} and \cite[Equation (3.2)]{CKW2}, we have
\[
b_{\tw^-}(\pi \tw, \gamma p)
= b_{\tw^-}(\pi \tw, p) + b_{\tw^-}(p,\gamma p)
= s(\tw) - b_{\tw^-}^\gamma;
\]
moreover, since $s(\phi^t \tw) = s(\tw) + t$ by \cite[Equation (3.8)]{CKW2}, we see that $\phi^t \tw \in \gamma_*\tB$ if and only if $s(\tw) - b_{\tw^-}^\gamma + t \in [0,\alpha]$, which proves \eqref{eqn:gB-int} and completes the proof of the lemma.
\end{proof}

\begin{proof}[Proof of Lemma \ref{lem:full-branch}]
By Lemma \ref{lem:s-interval}, the fact that $\tB_{\theta'}^{\epsilon^2} \cap \phi^{-t} \gamma_* \tB_\theta^\alpha \neq\emptyset$ implies existence of $\eta \in \Pa'$ such that
\[
(b_\eta^\gamma - t + [0,{\alpha}]) \cap [0,\epsilon^2] \neq \emptyset
\]
from which we deduce that
\[
b_\eta^\gamma - t - \epsilon^{\frac{3}{2}} + [0,{\alpha}+2\epsilon^{\frac{3}{2}}] \supset [0,\epsilon^2]
\]
By \cite[Lemma (4.11)]{CKW2}, it follows that every $\xi\in \Pa'$ has
\[
(b_\xi^\gamma - t - \epsilon^{\frac{3}{2}} + [0,\alpha+2\epsilon^{\frac{3}{2}}]) \cap [0,\epsilon^2] \neq\emptyset
\]
which in turn implies that
\[
b_\xi^\gamma - t - 2\epsilon^{\frac{3}{2}} + [0,{\alpha}+4\epsilon^{\frac{3}{2}}] \supset [0,\epsilon^2].
\]
By Lemma \ref{lem:s-interval}, this completes the proof.
\end{proof}

\begin{proof}[Proof of Lemma \ref{lem:scaling}]
By definition of $\hm$, we have $\hm(\hS^\gamma) = \overline \tm( \overline\tS^\gamma) = \epsilon^2 \bar\mu(\Pa\times\gamma\Fu)$. Then we need to prove that $\bar\mu(\Pa\times\gamma\Fu)=e^{\pm 4 h \epsilon} e^{-ht} \mu_p(\Pa')\mu_p(\Fu)$

Given $(\xi,\eta)\in\Pa'\times\gamma\Fu$, we can take $q$ to lie on a geodesic connecting $\xi$ and $\eta$, with $b_\xi(q,p)=0$; then we have
\[
|\beta_p(\xi,\eta)|:= |b_\xi(q,p) + b_\eta(q,p)|
\leq d(q,p) < \epsilon/2,
\]
where the last inequality uses \cite[Lemma 3.9]{CKW2}. Using this together with \eqref{eq:mub} gives
\[
\bar\mu(\Pa'\times\gamma\Fu) = e^{\pm h \epsilon/2} \mu_p(\Pa') \mu_p(\gamma\Fu),
\]
Using \cite[Proposition 5.1 (a)]{CKW1}  gives
\[
\mu_p(\gamma \Fu) =  \mu_{\gamma^{-1} p}(\Fu),
\]
and \cite[Proposition 5.1 (b)]{CKW1} gives
\[
\frac{d\mu_{\gamma^{-1} p}}{d\mu_p}(\eta) 
= e^{-h b_\eta(\gamma^{-1} p,p)}.
\]
When $\eta = \tc(-\infty)$, where $\tc:=c_{p,\gamma^{-1} p}$. Using the visibility condition as in the proof of Lemma \ref{lem:vis3}, for $t$ large enough $\eta\in\Fu'_{\theta'+\iota}$ for some $\iota>0$ very small. Using Lemma \ref{lem:vis3}, $b_\eta(p,\gamma p)= t\pm4\epsilon$. By \cite[Lemma 4.11]{CKW2}, for $\xi\in \Fu'$, $b_\xi(\gamma^{-1}p,p)$ varies by at most $\epsilon^2$. We conclude that $\mu_p(\gamma\Fu) = e^{\pm5\epsilon} e^{-ht} \mu_p(\Fu)$, and and this proves the lemma.
\end{proof}


\begin{thebibliography}{100}
\bibitem{BBB}
W. Ballmann, M. Brin and K. Burns, On surfaces with no conjugate points, J. Diff. Geom., 25 (1987) 249-273.
\bibitem{CKW1}
V. Climenhaga, G.  Knieper, K. War, 
Uniqueness of the measure of maximal entropy for geodesic flows on certain manifolds without conjugate points, arXiv:1903.09831
\bibitem{CKW2}
V. Climenhaga, G.  Knieper, K. War, 
Closed geodesics on surfaces without conjugate points, arXiv:2008.02249
\bibitem{FM}
A. Freire and R. Mane, On the Entropy of the Geodesic Flow in Manifolds Without Conjugate Points, Inventiones mathematicae 69 (1982) 375-392.
\bibitem{Hu} H. Huber, 
Über eine neue Klasse automorpher Funktionen und ein Gitterpunktproblem in der hyperbolischen Ebene. I, 
Comment. Math. Helv. 30 (1956), 20–62.
\bibitem{Ma1} G. Margulis,  Certain applications of ergodic theory to the investigation of manifolds of negative curvature,  Funkcional. Anal. i Priložen. 3 1969 no. 4, 89–90.
\bibitem{Ma2} G. Margulis,
 {\it On some aspects of the theory of Anosov systems}, Springer Monographs in Mathematics. Springer-Verlag, Berlin, 2004. vi+139 pp
 \bibitem{Ni1}  P. Nicholls, A lattice point problem in hyperbolic space. Michigan Math. J. 30 (1983), no. 3, 273–287.
  \bibitem{Ni2}  P. Nicholls,  {\it The Ergodic Theory of Discrete Groups}, LMS Lecture Note Series, 143, CUP. Cambridge, 1989
\bibitem{Sh} R. Sharp, Sector estimates for Kleinian groups,  Portugaliae Mathematica. Nova Série 58, no. 4 (2001): 461–71.
     \end{thebibliography}

      \end{document}